\documentclass[12pt]{article}%
\usepackage{amsmath}
\usepackage{amsfonts}
\usepackage{amssymb}
\usepackage{graphicx}%
 \makeatletter
\newfont{\footsc}{cmcsc10 at 8truept}
\newfont{\footbf}{cmbx10 at 8truept}
\newfont{\footrm}{cmr10 at 10truept}
\newtheorem{theorem}{Theorem}

\newtheorem{lemma}[theorem]{Lemma}

\newtheorem{proposition}[theorem]{Proposition}

\newenvironment{proof}[1][Proof]{\noindent{\textbf {#1}  }}  {\hfill$\Box$\bigskip}

\def\blfootnote{\xdef\@thefnmark{}\@footnotetext}
\begin{document}

\title{A lower bound for the sum of the two largest signless Laplacian eigenvalues}
\author{Leonardo de Lima\thanks{Department of Production Engineering, Federal
Center of Technological Education, Rio de Janeiro, Brazil;\textit{email:
llima@cefet-rj.br, leolima.geos@gmail.com}} \thanks{Research supported by CNPq
Grant 305867/2012--1 and FAPERJ 102.218/2013.} \ and$\;$Carla Oliveira\thanks{Department of
Mathematical and Estatistics, National School of Statistics, Rio de Janeiro RJ, Brazil ;
\textit{email: carla.oliveira@ibge.gov.br}} \thanks{Research supported by CNPq Grant 305454/2012--9.} }
\date{}
\maketitle

\begin{abstract}
Let $G$ be a graph of order $n \geq 3$ with sequence degree given as $d_{1}\left(  G\right) \geq \ldots \geq d_{n}\left(  G\right)$ and
let $\mu_1(G),\ldots, \mu_n(G)$ and $q_1(G), \ldots, q_{n}(G)$ be the Laplacian and signless Laplacian eigenvalues of $G$ arranged in non increasing order, respectively. Here, we consider the Grone's inequality [R. Grone, Eigenvalues and degree sequences of graphs, Lin. Multilin. Alg. 39 (1995) 133--136]
$$ \sum_{i=1}^{k} \mu_{i}(G)  \geq \sum_{i=1}^{k} d_{i}(G)+1$$ and prove that for $k=2$, the equality holds if and only if $G$ is the star graph $S_{n}.$
The signless Laplacian version of Grone's inequality is known to be true when $k=1.$ In this paper, we prove that it is also true for $k=2,$ that is,
$$q_{1}(G)+q_{2}(G) \geq d_1(G)+d_2(G)+1$$ with equality if and only if $G$ is the star $S_{n}$ or the complete graph $K_{3}.$  When $k \geq 3$, we show a counterexample.  \medskip



\end{abstract}

\section{Introduction and main results}

Define $G=(V,E)$ as a finite simple graph on $n$ vertices. The sequence degree of $G$ is denoted by $d(G) = \left(d_1(G),d_2(G),\ldots,d_n(G) \right)Ò$ such that $ d_1(G) \geq d_2 (G)\geq \ldots \geq d_n(G).$ Write $A$ for the adjacency matrix of $G$ and let $D$ be the diagonal matrix of the row-sums of $A,$ i.e., the degrees of $G.$
The matrix $L\left(  G\right)  =A-D$ is called the \emph{Laplacian} or the $L$-matrix of $G$ and the matrix $Q\left(  G\right)  =A+D$ is called the \emph{signless Laplacian} or the $Q$-matrix of $G$. As usual, we shall index the eigenvalues of $L\left(  G\right)  $ and $Q\left(  G\right)  $ in non-increasing order and denote them as $\mu_{1}(G) \geq \mu_{2}(G) \geq \ldots \geq \mu_{n}(G)$ and $q_{1}(G) \geq q_{2}(G) \geq \ldots \geq q_{n}(G),$ respectively.
We denote the following graphs on $n$ vertices: the complete graph $K_{n}$; the star $S_{n}$ and the complete bipartite graph $K_{n_1,n_2}$, such that $n_1 \geq n_2$ and $n=n_1+n_2.$

The main result of this paper is about to the sum of the two largest $Q-$eigenvalues of a graph. There is quite a few papers on that subject and a
contribution to this area have been made very recently by Ashraf \textit{et al.}, \cite{AOT13}, Oliveira \textit{et al.}, \cite{OLRC13}, and Li and Tian, \cite{LT14}. An useful result to the area was obtained by Schur \cite{Sc23} as the following:

\begin{theorem}[Schur's inequality, \cite{Sc23}]
Let $A$ be a real symmetric matrix with eigenvalues $\lambda_{1} \geq  \ldots \geq \lambda_{n}$ and diagonal elements $t_{1} \geq  \ldots \geq t_{n}.$ Then
$ \sum_{i=1}^{k} \lambda_{i} \geq  \sum_{i=1}^{k} t_{i} .$
\end{theorem}

In 1994, Grone and Merris, \cite{GM94}, proved that $ \mu_{1}(G) \geq d_1(G)+1$ with equality if and only if there exists a vertex of $G$ with degree $n-1.$ Based on Schur's inequality, Grone, in \cite{G95}, proved a more general bound to the Laplacian eigenvalues related to the sum of the vertex degrees:
\begin{equation}
 \sum_{i=1}^{k} \mu_{i} \geq  1 + \sum_{i=1}^{k} d_{i} . \label{in1}
\end{equation}
Considering Grone's inequality, \cite{G95}, we proved that the extremal graph to the case $k=2$ is the star graph $S_{n}.$ Hence, we state the equality conditions to the Grone's inequality when $k=2$ as presented in the following theorem.

\begin{theorem}\label{th4}
Let $G$ be a simple connected graph on $n \geq 3$ vertices. Then
$$ \mu_{1}(G)+\mu_{2}(G) \geq d_1(G)+ d_{2}(G) + 1 $$
with equality if and only if $G$ is a star $S_{n}.$
\end{theorem}

The signless Laplacian version of Grone's inequality could be stated as the following:
\begin{equation}
 \sum_{i=1}^{k} q_{i}(G) \geq  1+ \sum_{i=1}^{k} d_{i}(G). \label{in2}
\end{equation}

Motivated by Grone's inequality, we considered to study whether the signless Laplacian version of the inequality (\ref{in1}) is true.
The case $k=1$ has been proved in the literature (see Lemma \ref{lemmaq1} in Section 2). For $k \geq 3,$ we show a counterexample such that inequality (\ref{in2}) is not true when $G$ is the star graph plus one edge.

The main result of this paper proves that inequality (\ref{in2}) is true for $k=2$ as stated by the next theorem.

\begin{theorem}\label{th3} Let $G$ be a simple connected graph on $n \geq 3$ vertices. Then
\[
q_{1}\left(  G\right) + q_{2}\left(  G\right)  \geq d_{1}\left(  G\right) + d_{2}\left(  G\right) + 1.
\]
Equality holds if and only if $G$ is one of the following graphs: the complete graph $K_{3}$ or a star $S_{n}.$
\end{theorem}

The paper is organized such that preliminary results are presented in the next section and the main proofs are in the Section \ref{prova}.

\section{Preliminary results}

%


Define the graph $S_{n}^{+}$ as the graph obtained from a star $S_{n}$ plus an edge. We introduce the paper showing that inequality (\ref{in2}) does not hold for $k \geq 3.$
\begin{proposition}\label{pr1}
Let $G$ be isomorphic to $S_{n}^{+}.$ For $k \geq 3,$  $$\sum_{i=1}^{k} q_{i}(G) <  1+ \sum_{i=1}^{k} d_{i}(G).$$
\end{proposition}

\begin{proof}
Let $G$ be isomorphic to $S_{n}^{+}.$ In this case, $d_1(G)=n-1,d_2(G)=d_3(G)=2$ and $d_4(G)=\ldots=d_{n}(G)=1.$ From Oliveira \emph{et al.} \cite{OLRC13}, we have
$q_3(G)=\ldots=q_{n-1}(G) = 1$ and also
\begin{eqnarray}
n <& q_1(G) <& n + \frac{1}{n} \nonumber \\
3 - \frac{2.5}{n} <& q_2(G) <& 3 - \frac{1}{n}. \nonumber
\end{eqnarray}
From Das in \cite{Das10}, $q_{n}(G) < d_{n}(G)$ and then we get
\begin{eqnarray}
0 \leq& q_{n} (G)<& 1. \nonumber
\end{eqnarray}

Since for $k \geq 3,$
\begin{eqnarray}
1+\sum_{i=1}^{k} d_{i}(G) = n+k+1 \nonumber
\end{eqnarray}
and
\begin{eqnarray}
\sum_{i=1}^{k} q_{i}(G) < n+k+1, \nonumber
\end{eqnarray}
then the result follows.
\end{proof}

The following two results are important to our purpose here.

\begin{lemma}[\cite{CRS07}]\label{lemmaq1}
Let $G$ be a connected graph on $n \geq 4$ vertices. Then, $$q_{1}(G) \geq d_{1}(G)+1$$
with equality if and only if $G$ is the star $S_{n}.$
\end{lemma}

\begin{lemma}[\cite{Das10}]\label{lemmaq2}
Let $G$ be a graph. Then $$q_{2}(G) \geq d_{2}(G) -1.$$
\end{lemma}

From Lemmas \ref{lemmaq1} and \ref{lemmaq2}, it is straightforward that $q_{1}(G)+ q_2(G) \geq d_1(G) + d_2(G).$ Here, we improved that lower bound to $d_1(G)+d_2(G)+1$ and in order to prove it we need to define the class of graphs $\mathcal{H}(p,r,s)$ that is obtained from $2K_1 \vee \overline{K_{p}}$ with additional $r$ and $s$ pendant vertices to the vertices $u$ and $v$ with largest and second largest  degree, respectively (see Figure 1).

The Propositions \ref{pr4} and \ref{pr5} will be useful to prove Theorem \ref{th3} and both present a lower bound to $q_2(G)$ within the family $\mathcal{H}(p,r,s).$

\begin{figure}[h]
\begin{center}
\includegraphics[width=2in, height=1.5in]{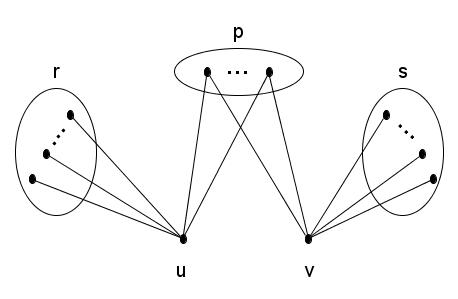}
\label{fig1}
\caption{Graph $\mathcal{H}(p,r,s)$.}
\end{center}
\end{figure}

\begin{proposition}\label{pr4}
For $p \geq 1$ and $r \geq s \geq 1,$  let $G$ be a graph on $n \geq 3$ vertices isomorphic to $\mathcal{H}(p,r,s).$ Then $$q_2(G) >
 d_2(G).$$
\end{proposition}

\begin{proof}
For $p \geq 1$, $r \geq s \geq 1$, consider $G$ as the graph isomorphic to $\mathcal{H}(p,r,s)$.  Labeling the vertices in a convenient way, we get

\[ Q(G) = \left( \begin{array}{c|c|c|c|c}
 p+r   &     0      &  \mathbf{1}_{1 \times p}   & \mathbf{1}_{1 \times r} & \mathbf{1}_{1 \times s}   \\  \hline
0       &    p+s    &   \mathbf{1}_{1 \times p}    & \mathbf{0}_{1 \times r}   & \mathbf{1}_{1 \times s} \\  \hline
 \mathbf{1}_{p \times 1} & \mathbf{1}_{p \times 1} & 2\mathbf{I}_{p \times p} & \mathbf{0}_{p \times r} & \mathbf{0}_{p \times s} \\ \hline
 \mathbf{1}_{r \times 1} & \mathbf{0}_{r \times 1} & \mathbf{0}_{r \times p} & \mathbf{I}_{r \times r} & \mathbf{0}_{r \times s} \\ \hline
 \mathbf{0}_{s \times 1} & \mathbf{1}_{s \times 1} & \mathbf{0}_{s \times p} & \mathbf{0}_{s \times r} & \mathbf{I}_{s \times s} \\
\end{array}\right ).
\]

Observe that $\mathbf{x}_{j} = e_{3}-e_{j}$, for $j=4,\ldots,p+2$ are eigenvectors associated to the eigenvalue $2$ which has multiplicity at least $p-1.$ Also, let us define $\mathbf{y}_{j} = e_{p+3}-e_{j}$ for each $j=p+4,\ldots,p+r+2$ and $\mathbf{z}_{j} = e_{p+r+3}-e_{j}$ for each $j=p+r+4,\ldots,p+r+s+2.$ Observe that $\mathbf{y}_{j}$ and $\mathbf{z}_{j}$ are eigenvectors associated to the eigenvalue $1$ with multiplicity at least $r+s-2.$ In all the cases the others $5$ eigenvalues are the same of the reduced matrix
$$M = \left(
\begin{array}{ccccc}
p+r &  0   &   p    &   r  & 0 \\
0  & p+s  &    p    &    0  &  s \\
1 &    1 &    2 &     0  & 0\\
1 &   0  &   0   & 1    &  0 \\
0 &   1  &  0   &   0   & 1 \\
\end{array}
 \right) .$$

The characteristic polynomial of M is given by $f(x,p,r,s) = x^{5}-(s+r+p+6)x^4+((r+p+3)s+(p+3)r+p^2+6p+5)x^{3}+((-2r-p-4)s+(-2p-2)r-2p^2-6p-2)x^{2}+((s+r)p+p^{2}+2p)x.$

Considering $r=s+k$, where $k \geq 0$ note that $f(d_2(G),p,s+k,s)=f(p+s,p,s+k,s) = (s+p) (ks^2 + s^2 + 2kps + 2p - 2ks - 2s + kp^2 - kp) > 0.$ As $f(0,p,r,s)<0$, if we take $q_2(G) < y < q_1(G),$ then $f(y,p,r,s)<0.$ So, since $f(d_2(G),p,r,s) > 0$ and from Lemma \ref{lemmaq2}, we get $q_2(G) > d_{2}(G).$

\end{proof}

\begin{proposition}\label{pr5}
For $p \geq 1$ and \ $r \geq 1,$ let $G$ be a graph on $n \geq 3$ vertices isomorphic to $\mathcal{H}(p,r,0).$ Then $$q_2(G) \geq d_2(G).$$ Equality holds if and only if $G=P_4$.
\end{proposition}

\begin{proof}
For $p,r \geq 1,$ consider $G$ as the graph isomorphic to $\mathcal{H}(p,r,0)$.  Labeling the vertices in a convenient way, we get

\[ Q(G) = \left( \begin{array}{c|c|c|c}
 p+r   &     0      &  \mathbf{1}_{1 \times p}   & \mathbf{1}_{1 \times r} \\  \hline
0       &    p    &   \mathbf{1}_{1 \times p}    & \mathbf{0}_{1 \times r}  \\  \hline
 \mathbf{1}_{p \times 1} & \mathbf{1}_{p \times 1} & 2\mathbf{I}_{p \times p} & \mathbf{0}_{p \times r}  \\ \hline
 \mathbf{1}_{r \times 1} & \mathbf{0}_{r \times 1} & \mathbf{0}_{r \times p} & \mathbf{I}_{r \times r}  \\
\end{array}\right ).
\]

If $p=r=1$, then $q_2(G) = d_2(G)=2$. If $p \geq 1$ and $r \geq 2$, observe that $\mathbf{x}_{j} = e_{3}-e_{j}$, for $j=4,\ldots,p+2$ are eigenvectors associated to the eigenvalue $2$ which has multiplicity at least $p-1.$

Let us define $\mathbf{y}_{j} = e_{p+3}-e_{j}$ for each $j=p+4,\ldots,p+r+2.$ Observe that $\mathbf{y}_{j}$ are eigenvectors associated to the eigenvalue $1$ with multiplicity
at least $r-1.$ The others $4$ eigenvalues are the same of the reduced matrix

$$M = \left(
\begin{array}{cccc}
p+r &  0   &   p    &   r \\
0  & p  &    p    &    0  \\
1 &    1 &    2 &     0  \\
1 &   0  &   0   & 1  \\
\end{array}
 \right) .$$

The characteristic polynomial of $M$ is given by $f(x,p,r) = x^{4}+(-r-2p-3)x^3+((p+2)r+p^2+4p+2)x^{2}+(-pr-p^2-2p)x.$ As $f(-1,p,r)>0$, if we take $q_2(G) < y < q_1(G),$ then $f(y,p,r)<0.$

Note that $f(d_2(G),p,r)=f(p,p,r) = rp^2 > 0.$ Therefore, from Lemma \ref{lemmaq2}, we have $q_2(G) > d_{2}(G)$. So $q_2(G) \geq d_2(G)$ with equality if and only if $G=P_4$.

\end{proof}

Let $\mathcal{G}(p,r,s)$ be the graph isomorphic to $\mathcal{H}(p,r,s)$ plus the edge $(u,v),$ see Figure 2. The next proposition shows that Theorem \ref{th3} is true for the family $\mathcal{G}(0,r,s).$
\begin{figure}[h]
\begin{center}
\includegraphics[width=2in, height=1.5in]{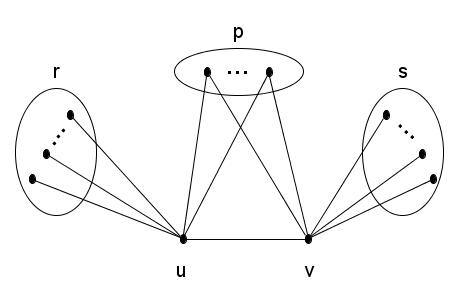}
\caption{Graph $\mathcal{G}(p,r,s)$.}
\end{center}
\label{fig1}
\end{figure}

\begin{proposition}\label{pr3}
For $r,s \geq 1,$ let $G$ be isomorphic to $\mathcal{G}(0,r,s).$  Then,
$$q_1(G) + q_{2}(G) > d_1(G) + d_{2}(G)+1.$$
\end{proposition}

\begin{proof}
For $r,s \geq 1,$ let $G$ be the graph isomorphic to $\mathcal{G}(0,r,s).$
Labeling the vertices of $G$ in a convenient way, we get

\[ Q(G) = \left( \begin{array}{c|c|c|c}
 r+1   &     1      &  \mathbf{1}_{1 \times r} & \mathbf{1}_{1 \times s}   \\  \hline
 1       &    s+1    &  \mathbf{0}_{1 \times r}   & \mathbf{1}_{1 \times s} \\  \hline
 \mathbf{1}_{r \times 1} & \mathbf{0}_{r \times 1} & \mathbf{I}_{r \times r} & \mathbf{0}_{r \times s} \\ \hline
 \mathbf{0}_{s \times 1} & \mathbf{1}_{s \times 1} & \mathbf{0}_{s \times r} & \mathbf{I}_{s \times s} \\
\end{array}\right ).
\]

Let us define $\mathbf{y}_{j} = e_{3}-e_{j}$ for each $j=4,\ldots,r+2,$  and $\mathbf{z}_{j} = e_{r+3}-e_{j}$
for each $j=r+4,\ldots,r+s+2.$ Observe that $\mathbf{y}_{j}$ and $\mathbf{z}_{j}$ are eigenvectors associated to the eigenvalue $1$ with multiplicity
at least $r+s-2.$ The others $4$ eigenvalues are the same of the reduced matrix
$$M = \left(
\begin{array}{cccc}
r+1 &  1   &  r  &  0 \\
1     & s+1  &  0  &  s \\
1     &   0  &  1  &  0\\
0     &   1  &  0  &  1 \\
\end{array}
 \right) .$$

The characteristic polynomial of $M$ is given by $f(x,r,s) = x^4+(-r-s-4)x^3+((r+2)s +2r+5)x^2 + (-r-s-2)x.$ As $f(0,r,s)>0$, if we take $q_2(G) < y < q_1(G),$ then $f(y,r,s)<0.$
Since $d_2(G)=s+1,$ we get
$f(d_2(G),r,s)=s(s+1)(r-s) \geq 0$ which implies $q_{2}(G) \geq d_{2}(G).$ From the equality conditions of Lemma \ref{lemmaq1}, $q_1(G) > d_1(G) + 1$ and the result follows.
\end{proof}

Next, in the Proposition \ref{pr2}, we present some bounds to $q_1(G)$ and $q_2(G)$ when $G$ is isomorphic to $\mathcal{G}(p,r,s)$ for $p \geq 1$, $r \geq s \geq 1.$

\begin{proposition}\label{pr2}
For $p \geq 1$, $r \geq s \geq 1,$ let $G$ be a graph on $n \geq 3$ vertices and isomorphic to $\mathcal{G}(p,r,s).$ Then

\begin{itemize}

\item[(i)] if $p = 1$ and $r=s$, then $q_1(G) > d_1(G) + \frac{3}{2}$ and $q_2(G) > d_2(G) -\frac{1}{2};$
\item[(ii)] if $p \geq 2$ and $r=s$, then $q_1(G) > d_1(G) + 2;$
\item[(iii)] if $p \geq 1$ and $r \geq s+3$, then $q_2(G) > d_2(G);$
\item[(iv)] if $p \geq 1$ and $r \in \{ s+1,s+2 \},$ then $q_1(G) > d_1(G)+1+\frac{p}{n}$ and $q_2(G) > d_2(G)-\frac{p}{n}.$

\end{itemize}
\end{proposition}

\begin{proof}
For $p \geq 1$, $r \geq s \geq 1$, let $G$ be isomorphic to the graph $\mathcal{G}(p,q,r).$ Labeling the vertices in a convenient way, we get
\[ Q(G) = \left( \begin{array}{c|c|c|c|c}
 p+r+1   &     1      &  \mathbf{1}_{1 \times p}   & \mathbf{1}_{1 \times r} & \mathbf{1}_{1 \times s}   \\  \hline
 1       &    p+s+1    &   \mathbf{1}_{1 \times p}    & \mathbf{0}_{1 \times r}   & \mathbf{1}_{1 \times s} \\  \hline
 \mathbf{1}_{p \times 1} & \mathbf{1}_{p \times 1} & 2\mathbf{I}_{p \times p} & \mathbf{0}_{p \times r} & \mathbf{0}_{p \times s} \\ \hline
 \mathbf{1}_{r \times 1} & \mathbf{0}_{r \times 1} & \mathbf{0}_{r \times p} & \mathbf{I}_{r \times r} & \mathbf{0}_{r \times s} \\ \hline
 \mathbf{0}_{s \times 1} & \mathbf{1}_{s \times 1} & \mathbf{0}_{s \times p} & \mathbf{0}_{s \times r} & \mathbf{I}_{s \times s} \\
\end{array}\right ).
\]

Observe that $\mathbf{x}_{j} = e_{3}-e_{j}$, for $j=4,\ldots,p+2$ are eigenvectors associated to the eigenvalue $2$ which has multiplicity at least $p-1.$ Also, let us define $\mathbf{y}_{j} = e_{p+3}-e_{j}$ for each $j=p+4,\ldots,p+r+2,$  and $\mathbf{z}_{j} = e_{p+r+3}-e_{j}$
for each $j=p+r+4,\ldots,p+r+s+2.$ Observe that $\mathbf{y}_{j}$ and $\mathbf{z}_{j}$ are eigenvectors associated to the eigenvalue $1$ with multiplicity
at least $r+s-2.$ The others $5$ eigenvalues are the same of the reduced matrix
$$M = \left(
\begin{array}{ccccc}
p+r+1 &  1   &   p    &   r  & 0 \\
1  & p+s+1  &    p    &    0  &  s \\
1 &    1 &    2 &     0  & 0\\
1 &   0  &   0   & 1    &  0 \\
0 &   1  &  0   &   0   & 1 \\
\end{array}
 \right) .$$

The characteristic polynomial of $M$ is given by $f(x,p,r,s) = x^{5}+(-s-r-2p-6)x^4+((r+p+4)s+(p+4)r+p^2+8p+13)x^{3}+((-2r-2p-5)s+(-2p-5)r-2p^2-14p-12)x^{2}+((p+2)s+(p+2)r+p^{2}+12p+4)x-4p.$ Since all eigenvalues of $Q$ are nonnegative, the roots of $f(x,p,r,s)$ are also nonnegative. As $f(0,p,r,s)<0$, if we take $q_2(G) < y < q_1(G),$ then $f(y,p,r,s)<0.$ This fact will be useful for the proof of the following cases below.

The largest and second largest degree of $G$ are given by $d_{1}(G) = p+r+1$ and $d_2(G) = p+s+1,$ respectively. Using the characteristic polynomial $f(x,p,r,s),$ we prove the following cases:

\vspace{0.5cm}

(i) $p=1$ and $r=s$:  observe that $f(d_1(G)+3/2,1,r,s) = -\frac{(6r+25)(4r^2+12r+25)}{32}$ and $f(d_2(G)-1/2,1,r,s) = \frac{(2r-1)(20r^2+12r+5)}{32}.$ For $r = s\geq 1$, we get $f(d_1(G)+3/2,1,r,s) <0 $ and $f(d_2(G)-1/2,1,r,s)>0.$  As from Lemma \ref{lemmaq1}, $f(d_1(G)+1,1,r,s) < 0$ and we get $f(d_1(G)+3/2,1,r,s) <0$, so $q_1(G) > d_1(G)+3/2.$ Also, from Lemma \ref{lemmaq2}, $f(d_2(G)-1,1,r,s)>0 $ and we get $f(d_2(G)-1/2,1,r,s)>0$, then $q_2(G)>d_2(G)-1/2.$

\vspace{0.5cm}


(ii) $p \geq 2$ and $r=s:$ note that $\;f(d_1(G)+2,p,r,r)= -(2r+3p+6)(pr-2r+p^{2}+p-2) <0$ and also from Lemma \ref{lemmaq1}, $\;f(d_1(G)+1,p,r,r)<0.$  So, we can conclude that $q_1(G) > d_1(G)+2.$

\vspace{0.5cm}

(iii) $p \geq 1$ and $r \geq s+3:$ note that $f(d_2(G),p,s+k,s)=ks^3+(3k-2)ps^2+((3k-5)p^2+(k+2)p-k)s+(k-3)p^3+(k+1)p^2 > 0 $ for $k \geq 3.$  Also, from Lemma \ref{lemmaq2}, we get $f(d_2(G)-1,p,s+k,s)>0.$  So, $q_2(G) > d_2(G).$

\vspace{0.5cm}

(iv) $p \geq 1$ and $r \in \{ s+1,s+2 \}:$ note that $n=2s+p+r+2$. Considering first $r=s+1,$ we get
$f(d_1(G)+1+p/(p+2s+4),p,s+2,s)= - (ns+np+p+n/n) <0$ and
$f(d_2(G)-p/(p+2s+4),p,s+2,s) = (2s^2+3ps+3s+p^2+2p)/(p+2s+4)^2 >0.$ Using Lemmas \ref{lemmaq1} and \ref{lemmaq2} analogous to the previous cases, we get $q_1(G)> d_1(G)+1+p/n$ and $q_2(G) > d_2(G)-p/n.$

Now, setting $r=s+2,$ we get $f(d_1(G)+1+p/(p+2s+4),p,s+2,s) <0$ and $f(d_2(G)-p/(p+2s+4),p,s+2,s)>0$ since
\begin{eqnarray}
& &f(d_1(G)+1+p/(p+2s+4),p,s+2,s) = -(32s^8+(208p+512)s^7+(568p^2+2992p+3456)s^6 + \nonumber \\
& &+ (836p^3+7120p^2+18048p+12800)s^5+(720p^4+8764p^3+36536p^2+59136p+28160)s^4+ \nonumber \\
& &+ (370p^5+5976p^4+36036p^3+98144p^2+113408p+36864)s^3+ \nonumber \\
& &+ (110p^6+2240p^5+18104p^4+72468p^3+145280p^2+126720p+26624)s^2+ \nonumber \\
& &+ (17p^7+420p^6+4326p^5+23576p^4+71008p^3+112000p^2+75776p+8192)s+ \nonumber \\
& &+ p^8+29p^7+368p^6+2616p^5+11024p^4+26960p^3+34944p^2+18432p)/(32s^5+\nonumber \\
& &+ (80p+320)s^4+(80p^2+640p+1280)s^3+\nonumber \\
& &+ (40p^3+480p^2+1920p+2560)s^2+(10p^4+160p^3+960p^2+2560p+2560)s+  \nonumber \\
& &+p^5+20p^4+160p^3+640p^2+1280p+1024) \nonumber
\end{eqnarray}
and
\begin{eqnarray}
& & f(d_2(G)-p/(p+2s+4),p,s+2,s) = (64s^8+(352p+640)s^7+(840p^2+3136p+2496)s^6+  \nonumber \\
& &+(1156p^3+6360p^2+11040p+4480)s^5+(1014p^4+7104p^3+18712p^2+19200p+2560)s^4+ \nonumber \\
& &+(581p^5+4814p^4+16220p^3+27048p^2+16640p-3072)s^3 +(211p^6+1998p^5+7832p^4+ \nonumber \\
& &+17128p^3+20400p^2+6144p-5120)s^2+(44p^7+470p^6+2044p^5+5120p^4+8720p^3+ \nonumber \\
& &+8288p^2+512p-2048)s+4p^8+48p^7+228p^6+600p^5 \nonumber \\
& & +1200p^4+2016p^3+1728p^2)/(p+2s+4)^5. \nonumber
\end{eqnarray}

Using Lemmas \ref{lemmaq1} and \ref{lemmaq2} analogously to the previous cases, we get $q_1(G)> d_1(G)+1+p/n$ and $q_2(G) > d_2(G)-p/n.$
\end{proof}

\section{Proofs}\label{prova}

In this section, we prove the main results of the paper.

\vspace{0.5cm}

\begin{proof}
\noindent \textbf{of Theorem \ref{th3}}
Let $G$ be a simple connected graph on $n \geq 3$ vertices. Assume that $u$ and $v$ are the vertices with largest and second largest degrees of $G$, i.e., $d(u) = d_{1}(G)$ and $d(v) = d_{2}(G).$
Take $H$ as a subgraph of $G$ containing $u$ and $v$ isomorphic to $\mathcal{H}(p,q,r)$ or $\mathcal{G}(p,r,s).$ Note that $d_1(G)+d_2(G)=d_1(H)+d_2(H)$ and from interlace Theorem (see \cite{HoJo85}),
$q_1(G)+q_2(G) \geq q_1(H)+q_2(H).$


Firstly, suppose that $H$ is isomorphic to $\mathcal{H}(p,r,s).$ Since $G$ is connected, the cases $p=0$ with any $r$ and $s$ are not possible.
If $p=1$ and $r=s=0,$ then $H = \mathcal{H}(1,0,0)= S_{3}$ and $q_1(H)+q_2(H)=4= d_1(H)+ d_2(H)+1.$ If $p \geq 2$ and $r=s=0,$ then $H = \mathcal{H}(p,0,0)= K_{2,p}$ and $q_1(H) +q_2(H) = 2p+2 > d_1(H)  + d_2(H)+1 = 2p+1.$ If $p,r \geq 1$ and $s=0,$ from Proposition \ref{pr5} and Lemma \ref{lemmaq1}, we get $q_1(H)+q_2(H)>d_1(H)+d_2(H)+1.$ Now, if $p \geq 1$ and $r \geq s \geq 1,$ from Proposition \ref{pr4} and Lemma \ref{lemmaq1}, follows that $q_1(H)+q_2(H)>d_1(H)+d_2(H)+1.$

Now, suppose that $H$ is isomorphic to $\mathcal{G}(p,q,r).$ If $p=s=0$ and $r \geq 1,$ $H = \mathcal{G}(0,r,0) = S_{r+2}$ and $q_1(H)+q_2(H)=r+3 = d_1(H)+d_2(H)+1.$  If $p=0$ and $r,s \geq 1,$ the result follows from Proposition \ref{pr3}. If $p=1$ and $r=s=0,$ then $H$ is the complete graph $K_3$ and $q_1(H)+q_2(H)=5 = d_1(H)+d_2(H)+1.$ If $p \geq 2$ and $r=s=0$, then $H = \mathcal{G}(p,0,0) = K_2 \vee \overline{K_{p}}$, i.e., the complete split graph, and it is well-known that $q_1(H)=(n+2+\sqrt{n^2+4n-12})/2$ and $q_{2}(H)=n-2.$ It is easy to check that for $p \geq 2$, we have $q_1(H)+ q_{2}(H) > d_{1}(H) + d_{2}(H)+1.$
If $p \geq 1, r \geq 1$ and $s \geq 0$, from the interlacing theorem (see \cite{HoJo85}) and the proof to the graph $\mathcal{H}(p,r,s)$, we get
$q_1(H)+q_2(H) \geq q_1(\mathcal{H}(p,r,s))+q_2(\mathcal{H}(p,r,s)) > d_1(\mathcal{H}(p,r,s))+d_2(\mathcal{H}(p,r,s))+1$ and the result of the theorem follows.

From the cases above, the equality conditions are restricted to the graphs $K_3$ and $S_n$ and the result follows.
\end{proof}

\vspace{0.5cm}

\begin{proof}
\noindent \textbf{of Theorem \ref{th4}}
Let $G$ be a simple connected graph on $ n\geq 3$ vertices. The result $\mu_1(G)+\mu_2(G) \geq d_1(G)+d_2(G)+1$ follows from Grone in \cite{G95}. Now, we need to prove the equality case.
Assume that $u$ and $v$ are the vertices with largest and second largest degrees of $G$, i.e., $d(u) = d_{1}(G)$ and $d(v) = d_{2}(G).$ Take $H$ as a subgraph of $G$ containing $u$ and $v$ isomorphic to $\mathcal{H}(p,q,r)$ or $\mathcal{G}(p,r,s).$ Note that $d_1(G)+d_2(G)=d_1(H)+d_2(H)$ and from interlace Theorem (see \cite{HoJo85})  $\mu_1(G)+\mu_2(G) \geq \mu_1(H)+\mu_2(H).$

Firstly, suppose that $H$ is isomorphic to $\mathcal{H}(p,r,s).$ In this case, $H$ is bipartite and $\mu_{i}(H) = q_{i}(H)$ for $i=1,\ldots,n$ (see \cite{CRS07}).  The proof is analogous to the Theorem \ref{th3} and the equality cases are similar, i.e., when $H = S_{3}.$

Now, suppose that $H$ is isomorphic to $\mathcal{G}(p,r,s).$ If $p=1, r=s=0$ then $H = \mathcal{G}(1,0,0)=K_3$ and $6 = \mu_1(H)+\mu_2(H) > d_1(H)+d_2(H)+1 = 5.$ If $p \geq 2, r=s=0,$ then $2p+4 = \mu_1(H)+\mu_2(H) > d_1(H)+d_2(H)+1 = 2p+1.$ The remanning cases are similar to the ones of the Theorem \ref{th3} and equality holds when $G = S_{n}.$
\end{proof}

\noindent \textbf{Acknowledgements:} Leonardo de Lima has been supported by CNPq Grant 305867/2012-1 and FAPERJ Grant 102.218/2013 and Carla Oliveira has also
been supported by CNPq Grant 305454/2012-9.

\end{document}